\documentclass{amsart}

\usepackage{amsthm, amssymb, amsmath, amscd}
\usepackage{eurosym}
\usepackage{amsfonts}
\usepackage{amsmath}
\usepackage{setspace}
\usepackage{graphicx}
\usepackage[all,cmtip]{xy}
\usepackage{indentfirst}

\newtheorem{theorem}{Theorem}[section]

\newtheorem{cor}[theorem]{Corollary}

\newtheorem{lemma}[theorem]{Lemma}

\newtheorem{prop}[theorem]{Proposition}

\theoremstyle{definition}
\newtheorem{define}[theorem]{Definition}
\newtheorem{ex}[theorem]{Example}
\newtheorem{remark}[theorem]{Remark}

\newcommand{\field}[1]{\mathbb{#1}}

\newcommand{\zkz}{\field{Z}/k\field{Z}}

\def\revdots{\mathinner{\mkern1mu\raise1pt\vbox{\k ern7pt\hbox{.}}\mkern2mu\raise4pt\hbox{.}\mkern2mu \raise7pt\hbox{.}\mkern1mu}}

\begin{document}
\title[$K$-homology with coefficients]{Geometric $K$-homology with coefficients I: \\ $\zkz$-cycles and Bockstein sequence}
\author{Robin J. Deeley}
\thanks{}
\date{}
\begin{abstract}
We construct a Baum-Douglas type model for $K$-homology with coefficients in $\zkz$.  The basic geometric object in a cycle is a ${\rm spin^c}$ $\zkz$-manifold.  The relationship between these cycles and the topological side of the Freed-Melrose index theorem is discussed in detail.  Finally, using inductive limits, we construct geometric models for $K$-homology with coefficients in any countable abelian group.  
\end{abstract}
\maketitle

The development of $K$-homology, the dual of $K$-theory, has involved both geometric and analytic ideas.  To briefly review the history, it was Atiyah who first proposed a model for $K$-homology using Fredholm operators in \cite{A}.  This was realized (independently) by the works of Kasparov \cite{Kas} and BDF-theory \cite{BDFear} and \cite{BDF}.  In these cases, the cycles are analytic in nature and are based on ideas from the theory of operator algebras.  Later, Baum and Douglas \cite{BD} defined a geometric model for $K$-homology.
\par
Recall that a cycle in the Baum-Douglas model is a triple, $(M,E,f)$, where $M$ is a compact ${\rm spin^c}$-manifold, $E$ is a smooth Hermitian vector bundle over $M$, and $f$ is a continuous map from $M$ to the space whose $K$-homology we are modeling.  For any homology theory, it is useful to study the corresponding theory with coefficients.  The easiest way to define this, in the context of geometric $K$-homology, is to alter the map $f$.  Thus, we define a cycle as before except for the map, which now maps from $M$ to $X\times Y$, where $X$ is the space whose $K$-homology is to be modeled while $Y$ is a space with $K_0(Y)=G$ and $K_1(Y)=0$.  Another approach is to alter the vector bundle $E$.  Since vector bundles determine classes in $K$-theory, the idea is to replace $E$ with a class in $K^*(M;G)$ (here, $K^*(M;G)$ denotes $K$-theory with coefficients in $G$).  This method has been developed by Jakob in \cite{Jak} and by Emerson and Meyer in \cite{EM} under the condition that $K^*(\cdot;G)$ is a multiplicative cohomology theory.  Both these methods are very general; to gain a concrete description of the cycles which define $K$-homology with coefficients, we required either a clear understanding of the space $Y$ or of the group $K^*(M;G)$.  
\par
There is another approach.  Namely, one could look to alter the ${\rm spin^c}$ manifold in the Baum-Douglas cycles.  To do so, one would need a different geometric object that is related to the specific coefficient group.  Because of this requirement, the model would not be as easy to define in general.  However, the model would be more intrinsic (i.e., would not rely on an understanding of $K$-theory with coefficients or of the space $Y$ above).  Moreover, the resulting model conceptualizes index theory for the geometric objects which are used to determine cycles.  For example, in the case of $\zkz$-coefficients and our geometric cycles, the relevant index theorem is the Freed-Melrose index theorem (see \cite{FM}).  
\par
The main goal of this paper is the construction of a geometric model for the coefficient group $\zkz$.  As such, the main result is the construction of the Bockstein sequence (also called the universal coefficient sequence) for this group (see Theorem \ref{Bockstein}).  The geometric object we use are ${\rm spin^c}$ $\zkz$-manifolds.  These singular spaces were introduced by Sullivan (e.g., \cite{MS}, \cite{Sul}) to study geometric topology.  Later, in \cite{Fre}, Freed began the study of index theory for such objects.  In \cite{FM}, Freed and Melrose proved an index theorem for $\zkz$-manifolds, which takes values in $\zkz$.  This result, along with Sullivan's work relating $\zkz$-manifolds to bordism groups with coefficients in $\zkz$, are the main reasons $\zkz$-manifolds are the correct object to determine cycles in our model.  \par
In fact, we define two geometric models for K-homology with coefficients in $\zkz$.  The first is the natural ``$\zkz$-version" of the original Baum-Douglas model while the second is the natural ``$\zkz$-version" of the geometric model for K-homology defined in \cite{Rav}.  In the latter model, the vector bundle in the original Baum-Douglas model is replaced by a K-theory class.  This change to the cycles allows for a concise realization of trivial cycles; this is an invaluable tool used in the proof of the Bockstein exact sequence (see Theorem \ref{Bockstein}).  These models are discussed in more detail at the start of Section 2.2.
\par
To give some context to our construction, we review the history of the original motivation for the Baum-Douglas model.  The starting point for this theory is the relationship between bordism and K-theory.  In particular, for a finite CW-complex, $X$, Conner-Floyd (see \cite{CF}) constructed a natural isomorphism
$$MU_{\rm even}(X)\otimes_{MU_{\rm even}(pt)} \field{Z} \rightarrow K_0(X)$$
where $MU$ denotes the bordism group of stably almost complex manifolds and $K_*(X)$ denotes the K-homology of $X$.  Moreover, Atiyah (see \cite{A}) showed that there is a natural surjection
$$MU_{\rm even/odd}(X) \rightarrow K_*(X)$$
Thus, one would naturally ask if there exists a more refined equivalence relation which turns this map into an isomorphism.  The Baum-Douglas model for K-homology answers this question in the affirmative and, moreover, with a relation defined in a very natural way.   
\par
In the context of $\zkz$-coefficients, we have a similarly defined surjective map
$$MU_{\rm even/odd}(X;\zkz) \rightarrow K_*(X;\zkz)$$
where the domain of this map is Baas-Sullivan bordism theory (in the particular case of $k$-points) and the image is K-homology with coefficients in $\zkz$.  Again, one is led to ask if there is a refined equivalence relation which turns this map into an isomorphism.  This question is answered in the affirmative in this paper.  Moreover, the relation defined should be as ``similar as possible" to the relation defined by Baum and Douglas.  \par
This last (somewhat informal) statement could be interpreted as the requirement that the diagram \\
\begin{center}
$\begin{CD}
MU_{\rm even/odd}(X) @>>> K_*(X) \\
@VVV  @VVV \\
MU_{\rm even/odd}(X;\zkz) @>>> K_*(X;\zkz) \\ 
\end{CD}$
\end{center} \vspace{0.5cm}
is respected by the refined relation in the $\zkz$-theory.  The relationship between our construction and bordism is discussed in more detail on page 17.
\par
The content of the paper is as follows.  In Section 1, we review the basic properties of $\zkz$-manifolds.  This includes the generalization of a number of notions from manifold theory to $\zkz$-manifolds such as the Freed-Melrose index theorem mentioned above.  Much of this material is not new, but is introduced since geometric properties of $\zkz$-manifolds are of fundamental importance to our model.  In Section 2, we introduce the cycles which determine our model.  These cycles are triples of the form, $((Q,P),(E,F),f)$, where $(Q,P)$ is a compact ${\rm spin^c}$ $\zkz$-manifold, $(E,F)$ is a $\zkz$-vector bundle over $(Q,P)$, and $f$ is a continuous map from $(Q,P)$ to the space whose $K$-homology (with coefficients in $\zkz$) we are modelling.  The main result of this section is a proof (under the condition that $X$ is a finite CW-complex) that $K_*(X;\zkz)$ fits into the Bockstein exact sequence.  The case of $K_0(pt;\zkz)$ is discussed in detail; in particular, we discuss its relationship with the topological side of the Freed-Melrose index theorem.   Finally, in Section 3, we produce models for any countable abelian group using inductive limits.       
\par
A word or two on the exposition may be helpful to the reader.  We have tried to limit prerequisites to a good understanding of the Baum-Douglas model for $K$-homology (i.e., an understanding of the papers \cite{BD}, \cite{BDT}).  A nice modern source for this material is \cite{BHS}.  We have followed this reference and \cite{Rav} for matters related to the Baum-Douglas model, and have followed \cite{Hig} and \cite{MS} for the theory of $\zkz$-manifolds.  Section 1 covers the basics of $\zkz$-manifold theory which we require for our development.  The reader is directed to \cite{Hig} and \cite{MS} for more details on the generalizations of a number of notions from manifold theory to $\zkz$-manifold theory.  Moreover, the reader who is unfamiliar with the theory of $\zkz$-manifolds is encouraged to read Section 1 of \cite{MS} for a short, but illuminating introduction to the subject.  In fact, the theory we develop here is best described as a formulation of the ideas presented in \cite{MS} and \cite{Sul1} (in particular, Chapter 6 of \cite{Sul1}) into the context of cycles of the form developed by Baum and Douglas in \cite{BD}. 
\par
A word of caution to the reader unfamilar with $\zkz$-manifolds is in order.  There is no action of the group $\zkz$ on these objects.  The reference to the group $\zkz$ can be explained by the fact that (even dimensional, ${\rm spin^c}$) $\zkz$-manifolds naturally imbed into a space, $W$, with $K^0(W)\cong \zkz$ (see $\tilde{H}^{2n}_k$ in Definition \ref{HigSpa}).  As we discuss in detail, this imbedding leads to a $\zkz$-valued index.  This is completely analogous to the case of (even dimensional, ${\rm spin^c}$) manifolds and the topological side of the Atiyah-Singer index theorem.  Even in Section 3 when we discuss inductive limit constructions, the relationship with the group is through the operation of disjoint union and not through any ``$\zkz$ group action" on the boundary components.
\par
We have used the following notation.  Throughout, $X$ will denote a finite CW-complex.  The K-theory (with compact supports) of $X$ is denoted by $K^*(X)$, while its K-homology is denoted by $K_*(X)$.  If $M$ is a manifold, then we denote the disjoint union of $k$-copies of $M$ by $kM$.  If $E_1$ and $E_2$ are vector bundles over $M_1$ and $M_2$, then we use $E_1 \dot{\cup} E_2$ to denote the vector bundle (over $M_1\dot{\cup}M_2$) with fiber at $x$ given by $(E_1)_x$ if $x\in M_1$ and $(E_2)_x$ if $x\in M_2$.  We also use $k E$ as notation for $\dot{\cup}_{k\; times} E$.  We use similar notation for mappings.

\section{Preliminaries}
\subsection{$\zkz$-manifolds}
\par
In this section, we introduce $\zkz$-manifolds, which are the basic geometric objects used in our model.  For this model, we require $\zkz$-manifolds with a ${\rm spin^c}$-structure.  
\subsection{Definition and basic properties of $\zkz$-manifolds}
\begin{define} \label{zkzmfld}
Let $Q$ be an oriented, smooth compact manifold with boundary.  We assume that the boundary of $Q$, $\partial Q$, decomposes into $k$ disjoint manifolds, $(\partial Q)_1, \ldots, (\partial Q)_k$.  A $\zkz$-structure on $Q$ is an oriented manifold, $P$, and orientation preserving diffeomorphisms, $\gamma_i: (\partial Q)_i \rightarrow P$.  A $\zkz$-manifold is a manifold with boundary, $Q$, with a fixed $\zkz$-structure.  We denote this by $(Q,P,\gamma_i)$.  We sometimes drop the maps from this notation and denote a $\zkz$-manifold by $(Q,P)$. 
\end{define}
\begin{remark} \label{singzkz}
From the data, $(Q,P,\gamma_i)$, we can create a singular space.  To do so, we note that the diffeomorphisms, $\{\gamma_i\}_{i=1}^{k}$, induce a diffeomeorphism between $\partial Q$ and $P\times \zkz$.  The singular space is then created by collapsing each $\{x\}\times \zkz \in P\times \zkz$ to a point.  This singular space will (usually) be denoted by $\tilde{Q}$.
\end{remark}
In Definition \ref{zkzmfld}, we have assumed that $Q$ and $P$ are both compact.  We can also consider the case when $Q$ (or both $Q$ and $P$) are not compact.  We will refer to such objects as {\it noncompact} $\zkz$-manifolds. 
\par
Many concepts from differential geometry and topology have natural generalizations from the manifold setting to the $\zkz$-manifold setting.  The generalization of vector bundles to $\zkz$-vector bundles is prototypical.  A $\zkz$-vector bundle is a pair, $(E,F)$, where $E$ is a vector bundle over $Q$, $F$ is a vector bundle over $P$, and $E|_{\partial Q}$ decomposes into k copies of $F$.  To be more precise, the identification of (i.e., isomorphism between) $E|_{\partial Q}$ and the k-copies of $F$ is also considered part of the data.  Additionally, we have natural definitions of a $\zkz$-Riemannian metric, a $\zkz$-fiber bundle, a ${\rm spin^c}$-structure on a $\zkz$-vector bundle, and a ${\rm spin^c}$-structure on a $\zkz$-manifold. The reader can see \cite[Definition 3.1]{Hig} for further details.
\par
We also have natural definitions of differentiable maps between $\zkz$-manifolds, which leads to a notion of diffeomorphism between $\zkz$-manifolds.  We will often require such maps to preserve certain additional structure (for example, the ${\rm spin^c}$-structure).  
\begin{ex}
We consider the manifold with boundary, denoted by $Q$, given in Figure \ref{z3pic} and take $P=S^1$.  Then one can easily see that $(Q,P)$ has the structure of a $\field{Z}/3$-manifold. \label{z3ex}
\end{ex}
\begin{ex}
Any compact oriented manifold without boundary is a $\zkz$-manifold for any $k$.  To see this, we take $P=\emptyset$ and note that $(M,\emptyset)$ has the structure required by Definition \ref{zkzmfld}. 
\end{ex}
Using the process described in Remark \ref{singzkz}, we can think of a $\zkz$-manifold as a singular space.  Then for any point in the singular space, there is a neighbourhood that is either diffeomorphic to a neighbourhood in $\field{R}^n$ or is of the form shown in Figure \ref{paperzkzloc}.  The number of ``sheets of paper" is equal to $k$.

\begin{figure}
    \centering
        \includegraphics[width=8cm, height=5cm]{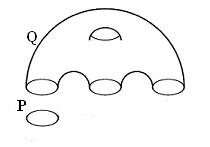} 
    \caption{$\field{Z}/3$-manifold from Example \ref{z3ex}.}
    \label{z3pic}
\end{figure}
\begin{figure}
    \centering
        \includegraphics[width=8cm, height=2.5cm]{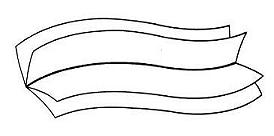} 
    \caption{Local picture of a $\field{Z}/4$-manifold.}
    \label{paperzkzloc}
\end{figure}

We now consider embeddings of $\zkz$-manifolds.  Throughout, all embeddings of manifolds with boundary will be neat embeddings (see Section 1.4 of \cite{Hir}).  
\begin{define}
Let $k$ and $n$ be fixed natural numbers.  Then, inside $\field{R}^n$, we let
\begin{eqnarray*}
H= \{(x_1, \ldots x_n) \in \field{R}^n \: | \: x_1 > 0 \} \\
H_0=\{ (x_1, \ldots x_n) \in \field{R}^n \: | \: x_1 = 0 \}
\end{eqnarray*}
Moreover, let $H^n_k$ be the space obtained by adjoining to $H$, $k$ disjoint, relatively open, unit radius disks in $H_0$.  Then $(H^n_k,\field{D}^{n-1})$ has the structure of a noncompact $\zkz$-manifold where we note that $\field{D}^{n-1}$ denotes the {\it open} unit disk.  We will denote the singular space constructed using the process described in Remark \ref{singzkz} by $\tilde{H}^n_k$. \label{HigSpa}
\end{define}
\begin{ex} \label{norBunZkz}
Let $(Q,P,\gamma_i)$ be a $\zkz$ manifold and $(H^{2N}_k,\field{D}^{2N-1})$ be the noncompact $\zkz$-manifold from Definition \ref{HigSpa}.  If we take $N$ large enough, then we have compatible embeddings $f_Q:Q \hookrightarrow H^{2N}_k$ and $f_P:P \hookrightarrow \field{D}^{2N-1}$.  The compatibility that we require is that, for each $i$, $f_Q |_{(\partial Q)_i}=(f_P \circ \gamma_i)|_{(\partial Q)_i}$ where $\{(\partial Q)_i\}_{i=1}^{k}$ is the decomposition of the boundary of $Q$ in Definition \ref{zkzmfld}.  \par
We have the following exact sequences: 
\begin{eqnarray*}
0 \rightarrow TQ \rightarrow T(H^{2N}_k)|_{Q} \rightarrow N_Q \rightarrow 0 \\
0 \rightarrow TP \rightarrow T(\field{D}^{2N-1})|_P \rightarrow N_P \rightarrow 0
\end{eqnarray*}
where $N_Q$ and $N_P$ are the normal bundles associated to the embedding $f_Q$ and $f_P$ respectively.  We then have that $(N_Q,N_P)$ is a $\zkz$-vector bundle over $(Q,P)$.  
\end{ex}
\begin{define}
Let $(Q,P,\gamma_i)$ and $(\hat{Q},\hat{P},\hat{\gamma}_i)$ be two $\zkz$-manifolds.  The {\it disjoint union} of $(Q,P,\gamma_i)$ and $(\hat{Q},\hat{P},\hat{\gamma}_i)$ is given by $(Q\dot{\cup}\hat{Q}, P\dot{\cup}\hat{P},\gamma_i \dot{\cup}\hat{\gamma_i})$, where the disjoint union of the mappings, $\gamma_i$ and $\hat{\gamma_i}$, is given by the map defined by $\gamma_i$ for points in $Q$ and by $\hat{\gamma_i}$ for points in $\hat{Q}$. \label{zkzDisUn}
\end{define}

\subsection{Bordism of $\zkz$-manifolds}
We now discuss bordism for $\zkz$-manifolds.  This concept is due to Sullivan \cite{Sul} (also see \cite{Baa}).
\begin{define}
Let $\bar{Q}$ be an $n$-dimensional, oriented, smooth, compact manifold with boundary.  In addition, assume we are given $k$ disjoint, oriented embeddings of an $(n-1)$-dimensional, oriented, smooth, compact manifold with boundary, $\bar{P}$, into $\partial \bar{Q}$.  Using the same notation as Definition \ref{zkzmfld}, we denote this as a triple $(\bar{Q},\bar{P},\gamma_i)$ (or just $(\bar{Q},\bar{P})$) where $\{\gamma_i\}_{i=1}^k$ denote the $k$ disjoint oriented embeddings.  We refer to such a triple as a $\zkz$-manifold with boundary.  The boundary of such an object is defined to be $\partial \bar{Q} - {\rm int}(k\bar{P})$ where $k\bar{P}$ denotes the $k$ copies of $\bar{P}$ in $\partial \bar{Q}$.  We note that the boundary has a natural $\zkz$-manifold structure induced by identifying the $k$ copies of the boundary of $\bar{P}$ (see Remark \ref{zkzbound} for more on the boundary).   \label{zkzmfldbound}
\end{define}
\begin{remark}
If a $\zkz-$manifold $(Q,P)$ is the boundary of the $\zkz$-manifold with boundary, $(\bar{Q},\bar{P})$, then 
\begin{eqnarray}
\partial \bar{Q} & = & Q \cup_{\partial Q} ( k\bar{P} ) \label{borEq1} \\
\partial \bar{P} & = & P \label{borEq2}
\end{eqnarray}
\label{zkzbound}
\end{remark}  
\begin{ex}  Three examples of $\zkz$-manifolds with boundary are: 
\begin{enumerate}
\item A $\zkz$-manifold is a $\zkz$-manifold with empty boundary. 
\item Using the notation of Definition \ref{zkzmfldbound}, a manifold with boundary, $\bar{Q}$, with $\bar{P}=\emptyset$ is a $\zkz$-manifold with boundary.  Moreover, its boundary when considered as a $\zkz$-manifold is the same as its boundary when considered a manifold with boundary. 
\item For any oriented $M$, $(kM \times [0,1],-M)$ is a $\zkz$-manifold with boundary.  Moreover, its boundary is $kM$. 
\end{enumerate}
\label{manBou}
\end{ex}
\begin{define} \label{borZkz}
We say that two $\zkz$-manifolds, $(Q,P)$ and $(\hat{Q},\hat{P})$, are {\it bordant} if $(Q,P)\dot{\cup}(-\hat{Q},-\hat{P})$ is a boundary in the sense of Remark \ref{zkzbound}.  The $\zkz$-manifold with boundary, $(\bar{Q},\bar{P})$ in Remark \ref{zkzbound}, will be called a bordism between $(Q,P)$ and $(\hat{Q},\hat{P})$.  We will denote this by $(Q,P) \sim_{bor}  (\hat{Q},\hat{P})$.
\end{define}
\begin{prop}
The operation $\sim_{bor}$ is an equivalence relation.
\end{prop}
This result is due to Sullivan, but Baas proves this result for more general classes of manifolds with singularities in \cite{Baa}.  Sullivan also shows that the bordism relation on $\zkz$-manifolds leads to bordism groups with coefficients in $\zkz$.  This connection is fundamental to the construction of our model for $K$-homology with coefficients in $\zkz$.
\par
In our model, we consider $\zkz$-bordisms which preserve the additional ${\rm spin^c}$-structure we put on our $\zkz$-manifold.  The reader should assume that, for the rest of this paper, all bordisms and $\zkz$-bordisms are ${\rm spin^c}$-bordisms.

\subsection{Index theory for $\zkz-$manifolds} \label{zkzMfldIndexSec}
In this section, we discuss the Freed-Melrose index theorem for ${\rm spin^c}$ $\zkz$-manifolds.  A special case of this theorem was proved by Freed in \cite{Fre}.  The general case is treated in \cite{FM} (see also \cite{Hig} and \cite{Ros1}).  \par
We will be most interested in the topological side of the Freed-Melrose index theorem.  We define the topological index map using a noncompact space which plays the same role as Euclidean space in the Atiyah-Singer index theorem (see \cite{AS1}).  One could also (see \cite{Ros1} for details) use a Moore space to construct the topological index.  This process would be analogous to using spheres (rather than Euclidean space) in the case of the Atiyah-Singer index theorem.     
\par  
To begin, the reader should recall that from a $\zkz$-manifold, $(Q,P)$, we can form a singular space following the process described in Remark \ref{singzkz}.  We will denote this singular space by $\tilde{Q}$.  We note that throughout this paper we will work with K-theory with compact supports.  The reader should recall or note that the $K^0(\tilde{Q})$ can be realized as the Grothendieck group of the semigroup of $\zkz$-vector bundles over $(Q,P)$ and that given an embedding of one ${\rm spin^c}$ $\zkz$-manifold into another we get a wrong-way map between the K-theories of the associated singular spaces (see \cite{Fre} for details).    
\par
We will be interested in this map in the following case.  Let $(Q,P)$ be a ${\rm spin^c}$ $\zkz$-manifold with ${\rm dim}(Q)$ even and $(H^{2N}_k, \field{D}^{2N-1})$ be the noncompact $\zkz$-manifold from Definition \ref{HigSpa}.  Then, as discussed in Example \ref{norBunZkz}, for $N$ sufficiently large there is a neat embedding 
$$i: (Q,P) \hookrightarrow (H^{2N}_k, \field{D}^{2N-1})$$
The wrong-way map induced from $i$ will be denoted by $\pi^{\tilde{Q}}_{!}$.  \par
A standard computation in K-theory shows that 
\begin{eqnarray}
K^0(\tilde{H}^{2N}_k) & \cong &  \zkz  \label{MooreK0} \\
K^1(\tilde{H}^{2N}_k) & \cong &  0 \label{MooreK1}
\end{eqnarray} 
Since ${\rm dim}(Q)$ is even, we have that the range of $\pi^{\tilde{Q}}_{!}$ is $\zkz$. 
\begin{define}
Let $(Q,P)$ be a ${\rm spin^c}$ $\zkz$-manifold and $(E,F)$ be a $\zkz$-vector bundle over it.  We denote the Dirac operator on $(Q,P)$ twisted by $(E,F)$ by $D_{(E,F)}$ (see \cite{Ros1} Definition 2.4 for more on the Dirac operator).  Let  
$${\rm ind}_{\zkz}^{top}(D_{(E,F)}):=\pi^{\tilde{Q}}_{!}([E,F])\in K^0(\tilde{H}^{2N}_k)  \cong  \zkz$$
\label{FMindexGeo}
\end{define}
If $Q$ has odd dimension, then we can produce a topological index using similar methods.  However, in this case, the $\zkz$-topological index vanishes (see \cite{Fre}).  The next theorem is the Freed-Melrose index theorem for $\zkz$-manifolds.  It is analogous to the Atiyah-Singer index theorem for manifolds.
\begin{theorem}
Let $(Q,P)$ be a ${\rm spin^c}$ $\zkz$-manifold and $D$ a twisted Dirac operator on it.  Then 
\begin{equation*}
{\rm ind}_{\zkz}^{top}(D)={\rm ind}(D^{\rm APS}) \: {\rm mod} \: k
\end{equation*}
where ${\rm ind}(D^{\rm APS})$ denotes the Fredholm index of the twisted Dirac operator with the Atiyah-Patodi-Singer boundary conditions (see \cite{APS1} for more details). 
\end{theorem}
Proofs of this result can be found in \cite{FM}, \cite{Hig}, \cite{Ros1}, and \cite{Zha}. 
We will need two properties of this index.  Proofs of these properties follow (more or less directly) from results in \cite{Fre}. 
\begin{theorem} \label{borInv}
The $\zkz$-index is a ${\rm spin^c}$ $\zkz$ cobordism invariant.
\end{theorem} 
To state the other property of the Freed-Melrose index, we will need to introduce some notation.  Let $M$ be a fixed compact ${\rm spin^c}$-manifold without boundary.  The reader should recall that a fiber bundle over a $\zkz$-manifold $(Q,P)$ with fiber $M$, is a $\zkz$-manifold, $(E,F)$, and a pair of (compatible) fiber bundles, $\pi_Q^M: E \rightarrow Q$ and  $\pi_P^M: F \rightarrow P$ (where both $\pi_Q^M$ and $\pi_P^M$ have fiber $M$; see \cite{Hig} for details).  \par
Let $(Q,P)$ be a ${\rm spin^c}$ $\zkz$-manifold and $(E,F)$ be a fiber bundle over it.  Moreover, assume $(E,F)$ has a fixed ${\rm spin^c}$-structure which is compatible with both the ${\rm spin^c}$-structure on $(Q,P)$ and the ${\rm spin^c}$-structure on $M$.  Then, the direct image map in K-theory, 
$$\pi^{M}_{!}:K^*(\tilde{E}) \rightarrow K^*(\tilde{Q})$$
is well-defined and satisfies
\begin{equation}
\pi^{\tilde{E}}_{!}=\pi^{\tilde{Q}}_{!} \circ \pi^{M}_{!} \label{multZkzProp} 
\end{equation}
where $\pi^{\tilde{E}}_{!}$ and $\pi^{\tilde{Q}}_{!}$ are the $\zkz$-direct image map discussed above.  Since the $\zkz$-topological index is defined in terms of the direct image map, it also has this property (i.e., satisfies Equation \ref{multZkzProp}).  This property is the $\zkz$-version of the multiplicative property of the index (for closed manifolds) discussed in \cite{AS1}.  More details on the direct image map for $\zkz$-manifolds can be found in \cite{Fre} (see p. 246-247). 

\section{A model for $K_*(X; \zkz)$} 
\subsection{Definition of the model for $K_*(X;\zkz)$}
\begin{define}
Let $X$ be a compact Hausdorff space.  A $\zkz$-cycle (over $X$) is a triple, $((Q,P),(E,F),f)$, where $(Q,P)$ is a ${\rm spin^c}$ $\zkz$-manifold, $(E,F)$ is a smooth Hermitian $\zkz$-vector bundle over $(Q,P)$ and $f$ is a continuous map (in the $\zkz$-sense) from $(Q,P)$ to $X$. \label{zkzcyc}
\end{define}
Here, a continuous map from $(Q,P)$ to $X$ in the $\zkz$-sense is a pair of continuous maps $f_Q: Q \rightarrow X$ and $f_P: P \rightarrow X$ such that (in the notation of Definition \ref{zkzmfld}) $$f_Q|_{(\partial Q)_i}=f_P$$ 
for each $i=1,\ldots k$.  
\par 
Functions satisfying this property are in one-to-one correspondence with the continuous functions on the singular space described in Remark \ref{singzkz}.  Also, the manifolds $Q$ (and $P$) in a $\zkz$-cycle, $((Q,P),(E,F),f)$, may not be connected.  In fact, $Q$ (and $P$) can have components with differing dimensions.  Also, the vector bundles on different components many have differing fiber dimensions.
\begin{define}
Given a $\zkz$-cycle, $((Q,P),(E,F),f)$, we will denote its opposite by $-((Q,P),(E,F),f)=(-(Q,P),(E,F),f)$ where $-(Q,P)$ is the ${\rm spin^c}$ $\zkz$-manifold with the opposite ${\rm spin^c}$ structure (see Definition 4.8 in \cite{BHS} for more on the opposite ${\rm spin^c}$ structure).  \label{oppzkz}
\end{define}
We now define the operations and relations on $\zkz$-cycles.  The reader should note the similarity with the operations and relations defined on the cycles from the Baum-Douglas model (see Section 5 of \cite{BHS}).  
\begin{define}
Let $((Q,P),(E,F),f)$ and $((\hat{Q},\hat{P}),(\hat{E},\hat{F}),\hat{f})$ be $\zkz$-cycles.  Then the disjoint union of these cycles is given by the cycle
\begin{equation*}
((Q\dot{\cup}\hat{Q},P\dot{\cup}\hat{P}),(E\dot{\cup}\hat{E},F\dot{\cup}\hat{F}),f\dot{\cup}\hat{f})
\end{equation*}
where the disjoint union of $\zkz$-manifolds is defined in Definition \ref{zkzDisUn}, $E\dot{\cup} \hat{E}$ is the vector bundle with fibers given by $(E \dot{\cup} \hat{E})|_{x}= E|_x$ or $\hat{E}|_{x}$ (depending on whether $x\in E$ or $\hat{E}$), and $f \dot{\cup} \hat{f}$ is defined to be $f(x)$ if $x\in Q$ and $\hat{f}(x)$ if $x\in \hat{Q}$.
\end{define}
\begin{define} \label{zkzCycleBordism}
We say a $\zkz$-cycle, $((Q,P),(E,F),f)$, is a boundary if there exists 
\begin{enumerate}
\item a smooth compact ${\rm spin^c}$ $\zkz$-manifold with boundary, $(\bar{Q},\bar{P})$, 
\item a smooth Hermitian $\zkz$-vector bundle $(\bar{E},\bar{F})$ over $(\bar{Q},\bar{P})$, 
\item a continuous map $\Phi :(\bar{Q},\bar{P}) \rightarrow X$, 
\end{enumerate}
such that $(Q,P)$ is the $\zkz$-boundary of $(\bar{Q},\bar{P})$, $(E,F)= (\bar{E},\bar{F})|_{\partial (\bar{Q},\bar{P})}$, and $f=\Phi|_{\partial (\bar{Q},\bar{P})}$.  We say that $((Q,P),(E,F),f)$ is bordant to $((\hat{Q},\hat{P}),(\hat{E},\hat{F}),\hat{f})$ if 
\begin{equation*}
((Q,P),(E,F),f)\dot{\cup} -((\hat{Q},\hat{P}),(\hat{E},\hat{F}),\hat{f})
\end{equation*}
is a boundary.
\end{define}
We now define vector bundle modification for $\zkz$-cycles.   We consider the following setup.  Let $(Q,P)$ be a ${\rm spin^c}$ $\zkz$-manifold and $(W,V)$ be a ${\rm spin^c}$-vector bundle over $(Q,P)$ with even dimensional fibers.  We denote the trivial rank one real $\zkz$-vector bundle by $({\bf 1}_Q,{\bf 1}_P)$.  The $\zkz$-vector bundle, $(W\oplus {\bf 1}_Q, V\oplus {\bf 1}_P)$, is a ${\rm spin^c}$ $\zkz$-vector bundle.  Moreover, its total space is a noncompact $\zkz$-manifold and its components fit into the following exact sequences.
\begin{eqnarray*}
0 \rightarrow \tilde{\pi}_W^*(W\oplus {\bf 1}_Q) \rightarrow T(W\oplus {\bf 1}_Q) \rightarrow \tilde{\pi}_W^*(TQ) \rightarrow 0 \\
0 \rightarrow \tilde{\pi}_V^*(V\oplus {\bf 1}_P) \rightarrow T(V\oplus {\bf 1}_P) \rightarrow \tilde{\pi}_V^*(TP) \rightarrow 0
\end{eqnarray*}
where we have that $\tilde{\pi}_W:W\oplus {\bf 1}_Q \rightarrow Q$ and $\tilde{\pi}_V:V\oplus {\bf 1}_P \rightarrow P$.
By choosing compatible splittings, we have 
\begin{eqnarray*}
T(W\oplus {\bf 1}_Q) &  \cong & \tilde{\pi}_W^*(W\oplus {\bf 1}_Q) \oplus \tilde{\pi}_W^*(TQ) \\
T(V\oplus {\bf 1}_P) & \cong & \tilde{\pi}_V^*(V\oplus {\bf 1}_P) \oplus \tilde{\pi}_V^*(TP)
\end{eqnarray*}
This identification puts a ${\rm spin^c}$-structure on the $\zkz$-manifold given by the total space of $(W\oplus {\bf 1}_Q,V\oplus {\bf 1}_P)$.  Moreover, the ${\rm spin^c}$-structure is unique up to concordance (i.e., different splittings give concordant ${\rm spin^c}$-structures).  Finally, we denote the sphere bundles of $W\oplus {\bf 1}_Q$ and $V\oplus {\bf 1}_P$ by $Z_Q$ and $Z_P$ respectively.  We have a natural ${\rm spin^c}$ $\zkz$-structure, induced from $(W\oplus {\bf 1}_Q,V\oplus {\bf 1}_P)$, on $(Z_Q,Z_P)$. 
\begin{define} \label{vbModzkz}
Let $((Q,P),(E,F),f)$ be a $\zkz$-cycle and $(W,V)$ an even dimensional ${\rm spin^c}$ $\zkz$-vector bundle over $(Q,P)$.  Using the notation and results of the previous paragraphs, we have that $(Z_Q,Z_P)$, the sphere bundle of $(W\oplus {\bf 1}, V\oplus {\bf 1})$, is a ${\rm spin^c}$ $\zkz$-manifold.  Moreover, the vertical tangent bundle of $(Z_Q,Z_P)$, denoted by $(V_Q,V_P)$, is a ${\rm spin^c}$ $\zkz$-vector bundle over $(Z_Q,Z_P)$.  We then let $(S_{Q,V},S_{P,V})$ be the reduced spinor bundle associated to $(V_Q,V_P)$ and let $(\hat{E},\hat{F})$ be the even part of the dual of $(S_{Q,V},S_{P,V})$.  The vector bundle modification of $((Q,P),(E,F),f)$ by $(W,V)$ is the $\zkz$-cycle $((Z_Q,Z_P),(\hat{E} \otimes \pi^*(E),\hat{F}\otimes \pi^*(F)),f \circ \pi)$ where $\pi$ denotes the bundle projection.  We will also use the notation $((Q,P),(E,F),f)^{(W,V)}$ to denote the vector bundle modification of $((Q,P),(E,F),f)$ by $(W,V)$. 
\end{define}
\begin{remark} \label{remZkzVBM}
It is worth noting that $(Q,E,f)$ is a Baum-Douglas cycle with boundary and $(P,F,f|_P)$ is a Baum-Douglas cycle.  Moreover, the $\zkz$-vector bundle modification of $((Q,P),(E,F),f)$ by $(W,V)$ can be thought of as the Baum-Douglas vector bundle modification of the cycles $(Q,E,f)$ and $(P,F,f|_P)$ by $W$ and $V$ respectively.  \par
In particular, if we take a $\zkz$-cycle coming from a ${\rm spin^c}$-manifold without boundary, $M$, then a vector bundle modification in the $\zkz$ sense corresponds to a vector bundle modification in the sense of Baum-Douglas.   
\end{remark}
\begin{define} \label{zkzModel}
We define $K_*(X;\zkz)$ to be the set of equivalence classes of $\zkz$-cycles where the equivalence relation is generated by the following.
\begin{enumerate}
\item  If $((Q,P),(E_1,F_1),f)$ and $((Q,P),(E_2,F_2),f)$ are $\zkz$-cycles (with the same ${\rm spin^c}$ $\zkz$-manifold, $(Q,P)$, and map $f$), then 
\begin{equation*}
((Q\dot{\cup}Q, P\dot{\cup}Q),(E_1\dot{\cup}E_2,F_1\dot{\cup}F_2),f\dot{\cup}f) \sim ((Q,P),(E_1\oplus E_2,F_1\oplus F_2), f)
\end{equation*}
\item If $((Q,P),(E,F),f)$ and $((\hat{Q},\hat{P}),(\hat{E},\hat{F}),\hat{f})$ are bordant $\zkz$-cycles, then 
\begin{equation*}
((Q,P),(E,F),f) \sim ((\hat{Q},\hat{P}),(\hat{E},\hat{F}),\hat{f})
\end{equation*}
\item If $((Q,P),(E,F),f)$ is a $\zkz$-cycle and $(W,V)$ is an even-dimensional ${\rm spin^c}$-vector bundle over $(Q,P)$, then 
we define $((Q,P),(E,F),f)$ to be equivalent to the vector bundle modification (as described in Definition \ref{vbModzkz}) of this cycle by $(W,V)$.
\end{enumerate}
\end{define}
We note that the grading on $K_*(X;\zkz)$ is given as follows: $K_0(X;\zkz)$ (resp. $K_1(X;\zkz)$) is the set of equivalence classes of $\zkz$-cycles for which each component of $Q$ (recall that $Q$ is not necessarily connected) is even (resp. odd) dimensional.  
\begin{prop}
The set $K_*(X;\zkz)$ is a graded abelian group with the operation of disjoint union.  In particular, the identity element is given by the class of the trivial cycle (i.e., the cycle $(\emptyset,\emptyset,\emptyset)$) and the inverse of a cycle is given by its opposite cycle (see Definition \ref{oppzkz}).
\end{prop}
\begin{proof}
The operation of disjoint union clearly gives the structure of an abelian semigroup.  It is also clear that the trivial cycle is an additive identity.  To produce an inverse for a cycle, we note that any cycle which is a boundary represents the additive identity of the group and the union of any cycle with its opposite is a boundary.  In other words, the opposite of a cycle provides an additive inverse.
\end{proof}

\subsection{The Bockstein sequence}
We now define the Bockstein exact sequence for the groups, $K_0(X;\zkz)$ and $K_1(X;\zkz)$.  Our Bockstein exact sequence for K-homology is analogous to both the Bockstein exact sequence for bordism defined in \cite{MS} and the long exact sequence in K-homology (see \cite{BHS} or \cite{Rav}).  For the proof of exactness of the Bockstein sequence, we follow the proof that the Baum-Douglas model of relative $K$-homology has a long exact sequence.  The difficulty in the proof of the latter is the determination of concise conditions, in terms of the equivalence relations, which lead to a trivial Baum-Douglas cycle.  We face a similar problem here.  To overcome it, we use an idea of Jakob (see \cite{Jak} and also \cite{Rav}) and define the notion of normal bordism for $\zkz$-cycles.   Theorem \ref{norBorzkz} is a natural generalization of Corollary 4.5.16 of \cite{Rav} to the $\zkz$ setting.  \par
The model for K-homology defined in \cite{Rav} uses K-theory classes in place of vector bundles.  As such, to apply the results of \cite{Rav} directly, we need a slightly different model for $K_*(X;\zkz)$.  For cycles, we take $((Q,P),\varepsilon,f)$ where $(Q,P)$ and $f$ are as in Definition \ref{zkzcyc} and $\varepsilon$ is an element in $K^0(\tilde{Q})$.  Recall that $\tilde{Q}$ is the singular space associated with $(Q,P)$ (see Remark \ref{singzkz}) and that elements of $K^0(\tilde{Q})$ are given by formal difference of (isomorphism classes of) $\zkz$-vector bundles over $Q$.  The opposite of a cycle and disjoint union of cycles are defined in essentially the same way as in the previous section. \par  
The equivalence relation on these cycles is generated by bordism and vector bundle modification.  The former is defined as in Definition \ref{zkzCycleBordism} with the vector bundle data replaced by K-theory data.  For the latter, we use the $\zkz$-version of the definition of vector bundle modification in Section 4.2 of \cite{Rav}.  The lack of a disjoint union/direct sum relation is explained by the fact that this relation is contained in the relation generated by bordism and vector bundle modification.  For a proof of this fact in the Baum-Douglas setting, see Proposition 4.2.3 in \cite{Rav}; the proof given there generalizes to the $\zkz$-setting with only minor changes.  \par
It is important to keep track of which of the two models for K-homology and K-homology with coefficients in $\zkz$ are in use.  To ensure clarity, we denote K-theory classes exclusively by $\varepsilon$ and $\nu$ (with appropriate subscripts/superscripts).  Also, in Section \ref{norZkz}, the models in use will be the versions using K-theory classes rather than vector bundles.  The only other place in the paper where the models using K-theory classes are used is in the proof of the exactness of the Bockstein sequence (see Theorem \ref{Bockstein}).

\subsubsection{Normal bordism for manifolds and $\zkz$ manifolds} \label{norZkz}
The following notation is used in this section.  We denote the equivalence relation of the Baum-Douglas model or of our $\zkz$-model (see Definition \ref{zkzModel}), depending on context, by $\sim$ and the bordism relation by $\sim_{bor}$.  If $(M,\varepsilon,f)$ is a Baum-Douglas cycle and $V$ is a ${\rm spin^c}$-vector bundle over $M$ with even dimensional fibers, then the vector bundle modification of $(M,\varepsilon,f)$ by $V$ is written as: $$(M,\varepsilon,f)^V$$
Similar notation is used for the vector bundle modification of $\zkz$-cycles.   We recall that throughout $X$ denotes a finite CW-complex.       
\begin{define}
Let $M$ be a manifold and $E$ be a vector bundle over it.  Then $N_E$ is a complementary bundle for $E$, if $E\oplus N_E$ is a trivial vector bundle over $M$.  A normal bundle for $M$ will refer to a complementary bundle for $TM$. 
\end{define}
\begin{theorem}
Let $(M,\varepsilon,f)$ be a Baum-Douglas cycle over $X$.  Then $(M,\varepsilon,f)$ represents the zero element in $K_*(X)$ if and only if there exists a normal bundle, $N$, for $M$ such that $(M,\varepsilon,f)^N$ is a boundary.
\label{norBor}
\end{theorem}
We will need the following lemma which is a generalization of Lemma 4.4.3 in \cite{Rav}.  The proof given in \cite{Rav} generalizes without major change; the details are left as an exercise for the reader.  
\begin{lemma} \label{RavLemmaZkz}
Let $((Q,P),\varepsilon,f)$ be a $\zkz$-cycle.  Then for any even dimensional ${\rm spin^c}$ $\zkz$-vector bundles, $(E_0,F_0)$ and $(E_1,F_1)$, we have that 
\begin{equation*}
((Q,P),\varepsilon,f)^{(E_0\oplus E_1, F_0\oplus F_1)} \sim_{bor} (((Q,P),\varepsilon,f)^{(E_0,F_0)})^{p^*(E_1,F_1)}
\end{equation*} 
where $p$ denote the projection $(E_0,F_0) \rightarrow (Q,P)$.
\end{lemma}
\begin{define}
Let $(Q,P)$ be a $\zkz$-manifold and $(E,F)$ a $\zkz$-vector bundle over it.  Then $(E,F)$ is $\zkz$-trivial if $E$ is a trivial vector bundle over $Q$.  This implies that $F$ is trivial over $P$. \label{zkzNorBundle}
\end{define}
Of course, if $(Q,P)$ is a $\zkz$-manifold, then any trivial bundle over $Q$ has a natural $\zkz$-vector bundle structure.  
\begin{define}
Let $(Q,P)$ be a $\zkz$-manifold and $(E,F)$ a $\zkz$-vector bundle over it.  Then a $\zkz$-normal bundle for $(E,F)$ is a $\zkz$-vector bundle, $(N_E,N_F)$, over $(Q,P)$ such that $(E,F)\oplus (N_E,N_F)$ is $\zkz$-trivial (i.e., $E\oplus N_N$ is a trivial vector bundle).  In fact, we also require the trivialization to be compatible with the bundle maps associated to $(E,F)$ and $(N_E,N_F)$.  A $\zkz$-normal bundle for $(Q,P)$ will refer to a $\zkz$-normal bundle for $(TQ,T(P\times (0,1])|_P)$.  For notational convenience, we denote a $\zkz$-normal bundle for $(Q,P)$ simply as $N$. 
\end{define}
The existence of a $\zkz$-normal bundle for an arbitrary ${\rm spin^c}$ $\zkz$-manifold follows from Example \ref{norBunZkz}.
\begin{lemma} \label{stIsoNor}
Let $(Q,P)$ be a $\zkz$-manifold and $N_1$ and $N_2$ be two $\zkz$-normal bundles for $(Q,P)$.  Then there exist trivial $\zkz$-bundles $(E_1,F_1)$ and $(E_2,F_2)$ such that 
$$N_1 \oplus (E_1,F_1) \cong N_2 \oplus (E_2,F_2)$$
\end{lemma}
\begin{proof}
Let $(E_2,F_2)=N_1 \oplus (TQ,T(P\times (0,1])|_P)$ and $(E_1,F_1)=N_2 \oplus (TQ,T(P\times (0,1])|_P)$.  By the definition of $\zkz$-normal bundle (i.e., Definition \ref{zkzNorBundle}), $(E_1,F_1)$ and $(E_2,F_2)$ are both $\zkz$-trivial.  Moreover, 
$$N_1 \oplus (E_1,F_1) \cong  N_1 \oplus N_2 \oplus (TQ,T(P\times (0,1])|_P) \cong N_2 \oplus (E_2,F_2)$$
\end{proof}
\begin{define}
A $\zkz$-cycle, $((Q,P),\varepsilon,f)$, is said to normally bound if there exists an (even rank) $\zkz$-normal bundle, $N$, over $(Q,P)$, such that the $\zkz$-cycle, $((Q,P),\varepsilon,f)^{N}$, is a boundary.  Two $\zkz$-cycles are normally bordant if their difference normally bounds.
\end{define}
\begin{prop}
Normal bordism (denoted $\sim_{nor}$) defines an equivalence relation on $\zkz$-cycles.
\end{prop}
\begin{proof}
That $\sim_{nor}$ is reflexive follows from the existence of a normal bundle for any $\zkz$-manifold.  Symmetry is clear, so we need only show transitivity.  To this end, let $\{((Q_i,P_i),\varepsilon_i,f_i)\}_{i=0}^{2}$ be $\zkz$-cycles and $N_0$, $N_1$, $N_1^{\prime}$, and $N_2$ be $\zkz$-normal bundles, such that
\begin{eqnarray*}
((Q_0,P_0),\varepsilon_0,f_0)^{N_0} & \sim_{bor} & ((Q_1,P_1),\varepsilon_1,f_1)^{N_1} \\
((Q_1,P_1),\varepsilon_1,f_1)^{N_1^{\prime}} & \sim_{bor} & ((Q_2,P_2),\varepsilon_2,f_2)^{N_2}
\end{eqnarray*}
We now use the fact that $\zkz$-normal bundles are stably isomorphic (see Lemma \ref{stIsoNor} above).  Based on this result, there exist trivial $\zkz$-bundles, $\epsilon_1$ and $\epsilon_1^{\prime}$ (both over $(Q_1,P_1)$) such that $N_1\oplus \epsilon_1 \cong  N_1^{\prime}\oplus \epsilon_1^{\prime}$.  Let $\epsilon_0$ be the trivial $\zkz$-bundle over $(Q_0,P_0)$ of the same rank as $\epsilon_1$ and $\epsilon_2$ be the trivial $\zkz$-bundle over $(Q_2,P_2)$ of the same rank as $\epsilon_1^{\prime}$.  Since the vector bundle modification by trivial bundles extend across $\zkz$-bordisms, we have that 
\begin{eqnarray*}
(((Q_0,P_0),\varepsilon_0,f_0)^{N_0})^{p_0^*(\epsilon_0)} & \sim_{bor} & (((Q_1,P_1),\varepsilon_1,f_1)^{N_1})^{p_1^*(\epsilon_1)} \\
(((Q_1,P_1),\varepsilon_1,f_1)^{N_1^{\prime}})^{p_{1^{\prime}}^*(\epsilon_1^{\prime})} & \sim_{bor} & (((Q_2,P_2),\varepsilon_2,f_2)^{N_2})^{p_2^*(\epsilon_2)}
\end{eqnarray*}
Moreover, by applying Lemma \ref{RavLemmaZkz} (a number of times), we obtain
\begin{eqnarray*}
((Q_0,P_0),\varepsilon_0,f_0)^{N_0 \oplus \epsilon_0} & 
\sim_{bor} & ((Q_1,P_1),\varepsilon_1,f_1)^{N_1\oplus \epsilon_1} \\
& \sim_{bor} & ((Q_2,P_2),\varepsilon_2,f_2)^{N_2 \oplus \epsilon_2}
\end{eqnarray*}
The result now follows since $N_0 \oplus \epsilon_0$ and $N_2 \oplus \epsilon_2$ are both $\zkz$-normal bundles.
\end{proof}
We now show that the normal bordism equivalence relation is the same as the equivalence relation generated by disjoint union, bordism, and vector bundle modification (i.e., $\sim$).  It is clear that normal bordism implies equivalence with respect to the relation $\sim$.  The next two propositions show the converse.  
\begin{prop}
If $((Q,P),\varepsilon,f)$ is a $\zkz$-boundary, then it also $\zkz$-normally bounds.
\end{prop}
\begin{proof}
Let $((W,Z),\nu,g)$ be a $\zkz$-bordism with boundary, $((Q,P),\varepsilon,f)$.  We must show that $((Q,P),\varepsilon,f)$ normally bounds.  To do so, fix a $\zkz$-normal bundle, $N$, for $(W,Z)$ and then consider $((W,Z),\nu,g)^{N \oplus {\bf 1}}$.  The boundary of this cycle is $$((Q,P),\varepsilon,f)^{N|_Q \oplus 1},$$ 
but we have that 
$$TQ\oplus N|_Q \oplus {\bf 1} \cong TW|_Q \oplus N|_Q \cong (TW \oplus N)|_Q $$
where this last bundle is $\zkz$-trivial by assumption.  Hence, $N|_Q \oplus {\bf 1}$ is a $\zkz$-normal bundle for $(Q,P)$ and so $((Q,P),\varepsilon,f)$ normally bounds.
\end{proof}
\begin{prop}
If $((Q,P),\varepsilon,f)$ is a $\zkz$-cycle and $(V,W)$ is a $\zkz$-vector bundle with even-dimensional fibers, then $((Q,P),\varepsilon,f)^{(V,W)}$ is $\zkz$-normally bordant to $((Q,P),\varepsilon,f)$.
\end{prop}
\begin{proof}
We begin by constructing a normal bundle for $(Q^V,P^W)$.  To do so, we let $p:Q^V \rightarrow Q$ be the natural projection and note that 
$$T(Q^V) \oplus {\bf 1} \cong p^*(TQ \oplus V)\oplus {\bf 1}$$
Let $N$ be a $\zkz$-normal bundle for $(Q,P)$ and $V^c$ be a complement to $V$ (i.e., $V\oplus V^c$ is trivial).  Note that $V^c$ can be chosen to be a $\zkz$-bundle.  To summarize, $TQ \oplus N \cong \epsilon_Q$ and $V\oplus V^c = \epsilon_V$ where $\epsilon_Q$ and $\epsilon_V$ are trivial vector bundles. \par
Next, we consider 
\begin{eqnarray*}
T(Q^V) \oplus p^*(V^c \oplus N \oplus {\bf 1} ) & \cong & T(Q^V) \oplus {\bf 1} \oplus p^*(V^c \oplus N) \\
& \cong & p^*(TQ \oplus V)\oplus {\bf 1} \oplus p^*(V^c \oplus N) \\
& \cong & p^*(TQ \oplus N) \oplus p^*(V\oplus V^c) \oplus {\bf 1} \\
\end{eqnarray*}
This last bundle is trivial and hence $p^*(V^c \oplus N \oplus {\bf 1})$ is a normal bundle for $Q^V$. \par
Using Lemma \ref{RavLemmaZkz}, we have that 
\begin{eqnarray*}
(((Q,P),\varepsilon,f)^{(V,W)})^{p^*(V^c \oplus N \oplus {\bf 1})} & \sim_{bor} & ((Q,P),\varepsilon,f)^{(V,W) \oplus V^c \oplus N \oplus {\bf 1}} \\
& \sim_{bor} & ((Q,P),\varepsilon,f)^{\epsilon_V \oplus N \oplus {\bf 1}}
\end{eqnarray*}
The last modification is by the bundle $N \oplus \epsilon_V \oplus {\bf 1}$, which is a $\zkz$ normal bundle for $(Q,P)$; hence these cycles are normally bordant.  
\end{proof}
\begin{cor}
Let $X$ be a finite CW-complex.  Then a $\zkz$-cycle over $X$ represents the zero element in $K_*(X;\zkz)$ if and only if it $\zkz$-normally bounds. \label{norBorzkz}
\end{cor}
\subsubsection{Bockstein exact sequence}
We now construct and prove exactness of the Bockstein sequence for our model.  The reader should compare this sequence with both the one for bordism groups with coefficients in $\zkz$ found in \cite{MS} and the one for analytic K-homology with coefficients in $\zkz$ found in \cite{Sch}.  We briefly discuss the relationship between the various Bockstein sequences after proving the theorem.
\begin{theorem} \label{Bockstein}
Let $X$ be a finite CW-complex.  Then the following sequence is exact. \newline
\begin{center}
$\begin{CD}
K_0(X) @>k>> K_0(X) @>r>> K_0(X;\zkz) \\
@AA\delta A @. @VV\delta V \\
K_1(X;\zkz) @<r<<  K_1(X) @<k<< K_1(X) 
\end{CD}$
\end{center}
where the maps are 
\begin{enumerate}
\item $k : K_*(X) \rightarrow K_*(X)$ is given by multiplication by $k$. 
\item $r: K_*(X) \rightarrow K_*(X;\zkz)$ takes a cycle $(M,E,f)$ to $((M,\emptyset),(E, \emptyset),f)$.  
\item $\delta : K_*(X;\zkz) \rightarrow K_{*+1}(X)$ takes a cycle $((Q,P),(E,F),f)$ to $(P,F,f)$.   
\end{enumerate}
\end{theorem}
\begin{proof} 
We leave it to the reader to show that the stated result will follow from the corresponding result for the model using cycles of the form $((Q,P),\varepsilon,f)$.  The maps in this case are given by:
\begin{enumerate}
\item $k : K_*(X) \rightarrow K_*(X)$ is given by multiplication by $k$. 
\item $r: K_*(X) \rightarrow K_*(X;\zkz)$ takes a cycle $(M,\varepsilon,f)$ to $((M,\emptyset),\varepsilon,f)$.  
\item $\delta : K_*(X;\zkz) \rightarrow K_{*+1}(X)$ takes a cycle $((Q,P),\varepsilon,f)$ to $(P,\varepsilon_P,f)$.  Where, if $\varepsilon=[(E,F)]-[(E^{\prime},F^{\prime})]$, then $\varepsilon_P:=[F]-[F^{\prime}] \in K^0(P)$; $\varepsilon_P$ (as a class in K-theory) depends only on the class $\varepsilon$.   
\end{enumerate}
Our first goal is to show that these maps are well-defined.  It is clear from the fact that $K_*(X)$ is an abelian group that multiplication by $k$ is well-defined.  That the map $r$ is well-defined follows (essentially) from Item 2 of Example \ref{manBou} (for the bordism relation) and Remark \ref{remZkzVBM} (for vector bundle modification).  Finally, it follows from Remark \ref{zkzbound} Equation \ref{borEq2}, that the map $$\delta : K_*(X;\zkz)\rightarrow K_{*+1}(X)$$ respects the bordism relation.  The case of vector bundle modification is again covered by Remark \ref{remZkzVBM}. 
\par
We now prove that the composition of any two maps in the sequence is zero.  This uses the same ideas as the ones used in \cite{MS} to prove the exactness of the Bockstein sequence for bordism with coefficients in $\zkz$. 
\par
Firstly, 
\begin{equation*}
(r \circ k) (M,\varepsilon,f) = ((kM, \emptyset), (k\varepsilon, \emptyset), kf) 
\end{equation*}
That this cycle is trivial in $K_*(X;\zkz)$ follows from Item 3 of Example \ref{manBou}. The reader should note that one must also keep track of the K-theory class and continuous map involved.  \par
It is clear that $(\delta \circ r)=0$ since the image of cycles in $K_*(X)$ under $r$ have empty boundary.  We are left to show that $(k \circ \delta)=0$.  Consider 
\begin{equation*}
(k \circ \delta)((Q,P),[(E,F)]-[(E^{\prime},F^{\prime})],f)= (k P, k ([F]-[F^{\prime}]), k f|_P)
\end{equation*}  
We must show that this cycle is trivial in $K_*(X)$.  To do so, note that the Baum-Douglas cycle (with boundary), $(Q,[E]-[E^{\prime}],f)$, has boundary given by $(k P, k ([F]-[F^{\prime}]), k f|_P)$.  The bordism relation in the Baum-Douglas model implies that the cycle, $(k P, k ([F]-[F^{\prime}]), k f|_P)$, is trivial and hence $(k \circ \delta)=0$ . \par
To complete the proof of exactness, we must show that 
\begin{equation*}
{\rm ker}(k) \subseteq  {\rm im}(\delta),\ 
{\rm ker}(\delta)  \subseteq  {\rm im}(r), \
{\rm ker}(r) \subseteq  {\rm im}(k)
\end{equation*}
It is worth noting that only here do we need the notion of normal bordism.  Again, this is analogous to the case of the long term exact sequence in relative $K$-homology (see Section 4.6 of \cite{Rav}). \par
To begin, suppose that $(M,\varepsilon,f)$ is a Baum-Douglas cycle which is in ${\rm ker}(k)$.  We must construct a $\zkz$-cycle which maps to $[(M,\varepsilon,f)]$.  Corollary \ref{norBor} implies that there exists a normal bundle modification (we denote the normal bundle by $N$) such that $(kM,k\varepsilon,kf))^N$ is a boundary.  Moreover, we can choose $N$ so that $N|_M$ is well-defined and so that $k (M,\varepsilon,f)^{N|_M}$ is a boundary.  We now denote $N|_M$ simply as $N$.  By definition, there exists a manifold with boundary, $Q$, K-theory class in $K^0(Q)$, $\nu$, and map, $g$, such that 
\begin{eqnarray*}
\partial Q & = &  kM^N \\
\nu|_{\partial Q} & = & k \varepsilon^N \\
g|_{\partial Q} & = & kf
\end{eqnarray*}
Moreover, the class $\nu$ is in the image of the map on K-theory induced by the inclusion of $C(\tilde{Q}) \hookrightarrow C(Q)$.  Denote a preimage of $\nu$ by $\nu^{\prime}$.  We have therefore constructed a $\zkz$-manifold, $(Q,M^N)$, a class in $K^0(\tilde{Q})$, $\nu^{\prime}$ and a continuous map, $g:(Q,M^N)\rightarrow X$.  That is, $((Q,M^N),\nu^{\prime},g)$ is a $\zkz$-cycle.  Moreover, 
\begin{equation*}
\delta([((Q,M^N),\nu^{\prime},g)])=[(M,\varepsilon,f)^N ]=[(M,\varepsilon,f)]
\end{equation*}
\par
Next, suppose that $(M,\varepsilon,f)$ is a Baum-Douglas cycle such that $[(M,\varepsilon,f)]\in {\rm ker}(r)$.  This implies that $[((M,\emptyset),\varepsilon,f)]=0$ and, by Corollary \ref{norBorzkz}, that there exists a normal bundle modification (we denote the normal bundle by $N$) such that $((M,\emptyset),\varepsilon,f)^N$ is the boundary in the $\zkz$ sense.  That is, we have a $\zkz$-manifold with boundary, $(W,P)$, a K-theory class, $\nu$, in $K^0(\tilde{W})$, and a continuous map, $g$, such that 
\begin{eqnarray*}
\partial W & = &  M^N \dot{\cup} kP  \\
(\nu)|_{\partial W - (kP)} & = & \varepsilon^N   \\
g|_{\partial W} & = & f \dot{\cup} k(g|_P)
\end{eqnarray*}
Let $i^*$ denote the map on K-theory induced from the inclusion $C(\tilde{W})\hookrightarrow C(W)$.  Then $(W,i^*(\nu),g)$ defines a bordism between the Baum-Douglas cycles $(M^N,\varepsilon^N,f)$ and $(kP,i^*(\nu)|_{kP},g|_{kP})$.  Hence $[(M,\varepsilon,f)]$ is in the image of the map $k$.  \par
Finally we must show that if $((Q,P),\varepsilon,f)$ is a $\zkz$-cycle such that 
\begin{equation*}
\delta([((Q,P),\varepsilon,f)])=[(P,\varepsilon_P,f)]=0
\end{equation*}
then $[((Q,P),\varepsilon,f)]$ is in the image of $r$.  \par
To begin, we reduce to the case when $(P,\varepsilon_P,f|_P)$ is a boundary.  Theorem \ref{norBor} implies that there exists normal bundle modification (we again denote the bundle by $N$) such that $(P,\varepsilon_P,f)^N$ is a boundary.  We denote this bordism by $(W,\tilde{\varepsilon},\tilde{f})$; hence, $\partial W =P^N$.  Extending the normal bundle, $N$, from the manifold, $P$, to all of $Q$ is, in general, not possible.  However, since all normal bundles of $P$ are stably isomorphic (see Lemma \ref{stIsoNor}), we have a normal $\zkz$-vector bundle, $(N_Q,N_P)$, over $(Q,P)$ and a trivial bundle, $V$, such that $N_P=N\oplus V$.  Then
\begin{eqnarray*}
\delta(((Q,P),\varepsilon,f)^{(N_Q,N_P)}) & = & (P,\varepsilon_P,f)^{N_P} \\
& = & (P,\varepsilon_P,f)^{N\oplus V} \\
& \sim_{bor} & ((P,\varepsilon_P,f)^N)^{p^*(V)} \\
\end{eqnarray*}  
The cycle $((P,\varepsilon_P,f)^N)^{p^*(V)}$ is a boundary.   \par
Hence, without loss of generality, $(P,\varepsilon_P,f|_P)$ is a boundary.  Denote by $(\bar{P},\tilde{\varepsilon},g)$ the triple where $\bar{P}$ is a manifold with boundary with, $\partial \bar{P}=P$, $\tilde{\varepsilon}$ is a K-theory class in $K^0(\bar{P})$ with $\tilde{\varepsilon}|_{\partial \bar{P}}=\varepsilon$, and $g$ is a continuous function with $g|_{\partial \bar{P}}=f|_P$.  Form the manifold (without boundary) $M=Q \cup_{kP} k\bar{P}$ and also form the (singular space) $\tilde{M}=\tilde{Q}\cup_P \bar{P}$.  The fact that 
$$K^0(\tilde{M}) \rightarrow K^0(\tilde{Q})\oplus K^0(\bar{P}) \rightarrow K^0(\partial \bar{P})$$  
is exact in the middle implies that there exists $\nu \in K^0(\tilde{M})$ such that $\nu|_{\tilde{Q}}=\varepsilon$ and $\nu|_{\bar{P}}=\tilde{\varepsilon}$.  Moreover, it is clear that the continuous functions, $f$ on $Q$, and, $g$ on $\bar{P}$, are compatible on $\partial Q$ and $\partial (k\bar{P})$. Let $h$ denote the continuous map $f \cup_{kP} kg$.  Also let $i^*$ denote the map on K-theory induced from the map $i:C(\tilde{M}) \rightarrow C(M)$.  Then $(M,i^*(\nu),h)$ forms a Baum-Douglas cycle.  \par
We now show that $r(M,i^*(\nu),h)\sim((Q,P),\varepsilon,f)$ in $K_*(X;\zkz)$ by constructing a $\zkz$-bordism between them.  To this end, consider the Baum-Douglas bordism between $(M,i^*(\nu),h)$ and its opposite, $-(M,i^*(\nu),h)$.  Denoting this bordism by $(\hat{M},\tilde{\nu},\tilde{h})$, we have that 
\begin{eqnarray*}
\partial \hat{M} &  = &  M \dot{\cup} -M \\   
& = & M \dot{\cup} ( Q \cup_{\partial Q} k\bar{P})
\end{eqnarray*}
We create a $\zkz$-manifold with boundary from $\hat{M}$ by taking the $k$ embeddings of $\bar{P}$ into $\partial \hat{M}$.  The resulting $\zkz$-manifold, $(\hat{M},\bar{P})$, has $\zkz$-boundary, $M\dot{\cup} Q$.  Moreover, the K-theory class, $\tilde{\nu}$, and map, $\tilde{h}$, respect this $\zkz$-structure so that $((\tilde{M},\bar{P}),\tilde{\nu},\tilde{h})$ forms a $\zkz$-bordism between $((M,\emptyset),(V,\emptyset),h)$ and $((Q,P),\varepsilon,f)$.  Hence, $r(M,\phi^*(\nu),h)\sim((Q,P),\varepsilon,f)$ in $K_*(X;\zkz)$, completing the proof of the Bockstein exact sequence.
\end{proof}
The observant reader will note that this theorem is not quite sufficient to imply that $K_*(X;\zkz)$ is a realization of K-homology with coefficient in $\zkz$.  The issue is that it is not a priori clear that $K_*(X;\zkz)$ is a homology theory.  \par
A similar problem occurs in the development of the Baum-Douglas model for K-homology (see for example the discussion on p. 19 of \cite{BHS}).  To overcome this difficulty, Baum and Douglas construct a natural map from their geometric cycles to the analytic cycles of Kasparov and prove that it induces an isomorphism between the two theories (see \cite{BHS}).  Since analytic K-homology is known to be a homology theory the construction of this natural isomorphism implies that geometric K-homology is as well.
\par
In our setting, a similar construction is completed in \cite{Dee2}.  There, we construct a natural map from our geometric cycles to the analytic cycles of Schochet (see \cite{Sch}).  We then prove that (in the case of finite CW-complexes) this map induces an isomorphism between our geometric group and the analytic realization of K-homology with coefficients in $\zkz$ of Schochet.  Moreover, this isomorphism is natural with respect to the Bockstein sequences associated to these models.  In fact, the Bockstein sequence constructed in Theorem \ref{Bockstein} is a key part of the proof that the natural map between these theories is an isomorphism (see \cite{Dee2} for more details).  The end result of this construction is that $K_*(X;\zkz)$ (defined via our geometric cycles) does give a realization of K-homology with coefficients in $\zkz$.
\par
We now discuss the relationship between our construction, the Conner-Floyd map and ${\rm spin^c}$-bordism with coefficients.  Let $\Omega_{\rm{even/odd}}^{{\rm spin^c}}(X)$ denote the ${\rm spin^c}$-bordism group (similar remarks hold for complex bordism).  Then, at the level of cycles, the Conner-Floyd map (denoted by $\rm{CF}$) is given by:
$$(M,f)\in \Omega_{\rm{even/odd}}^{{\rm spin^c}}(X) \mapsto (M,M\times \field{C},f)\in K_*(X)$$
Sullivan's work on $\zkz$-manifolds implies that
$$\Omega_{\rm{even/odd}}^{{\rm spin^c}}(X;\zkz)= \{ ((Q,P),f) \}/\sim_{bor}$$
where $(Q,P)$ is a ${\rm spin^c}$ $\zkz$-manifold, $f:(Q,P)\rightarrow X$ is continuous, and $\sim_{bor}$ is the $\zkz$-bordism relation (see Definition \ref{borZkz}).  Based on this realization of ${\rm spin^c}$-bordism with coefficients in $\zkz$, there is a $\zkz$-version of the Conner-Floyd map, $\rm{CF}_{\zkz}$, which is explicitly defined at the level of cycles via
$$((Q,P),f)\in \Omega_{\rm{even/odd}}^{{\rm spin^c}}(X;\zkz) \mapsto ((Q,P),(Q\times \field{C},P\times \field{C}),f)\in K_*(X;\zkz)$$
Finally, the maps $CF$ and $CF_{\zkz}$ are natural with respect to the Bockstein sequences for the homology theories ${\rm spin^c}$-bordism and K-homology.  The proof of this fact follows from the explicit nature of the definitions of the maps involved.  For example, the commutativity of the diagram
\begin{center}
$\begin{CD}
\Omega^{{\rm spin^c}}_{\rm even}(X;\zkz) @>\rm{CF}_{\zkz}>> K_0(X;\zkz) \\
@V\delta_{bor} VV  @V\delta VV \\
\Omega^{{\rm spin^c}}_{\rm odd}(X) @>\rm{CF}>> K_1(X) \\ 
\end{CD}$
\end{center}
follows by direct calculation at the level of cycles (i.e., 
$$(\delta \circ \rm{CF}_{\zkz})([(Q,P),f])=[(P,P\times \field{C},f|_P)]=(\rm{CF} \circ \delta_{bor})([(Q,P),f])$$

\begin{ex}
In this example, we discuss $K_*(pt;\zkz)$.  The Bockstein sequence implies that the groups $K_0(pt;\zkz)$ and $K_1(pt;\zkz)$ are equal to $\zkz$ and $\{0\}$ respectively.  However, our goal here is to construct the natural map between the $\zkz$-cycles which generate $K_0(pt;\zkz)$ and $\zkz$.  \par
The isomorphism between $K_0(pt)$ and $\field{Z}$ is given by the map that takes even Baum-Douglas cycles to the topological index of the  Dirac operator associated to such a cycle.  The point is that the topological index is fundamental to the definition of the isomorphism between $K_0(pt)$ and $\field{Z}$.  Analogously, we will show here that the topological side of the Freed-Melrose index theorem is fundamental to the definition of the map between even $\zkz$-cycles over a point and $\zkz$. 
\begin{theorem}
Let $\Phi$ be the map on even dimensional $\zkz$-cycles over a point defined by $((Q,P),(E,F),f) \mapsto {\rm ind}^{top}_{\zkz}(D_{(E,F)})$, where ${\rm ind}^{top}_{\zkz}$ denotes the (topological) Freed-Melrose index and $D_{(E,F)}$ denotes the Dirac operator on $(Q,P)$ twisted by the $\zkz$-vector bundle $(E,F)$.  Then $\Phi$ descends to a group isomorphism between $K_0(pt;\zkz)$ and $\zkz$. \label{zkzPt}
\end{theorem}
\begin{proof}
We begin by proving that $\Phi$ descends to a group homomorphism.  That is, we show that $\Phi$ respects the equivalence relations on $K_0(pt;\zkz)$.  For the disjoint union relation, consider two $\zkz$-cycles of the form $((Q,P),(E_1,F_1),f)$ and $((Q,P),(E_2,F_2),f)$.  Then, it is clear that 
\begin{equation*}
{\rm ind}_{\zkz}(D_{(E_1\dot{\cup}E_2,F_1\dot{\cup}F_2)})={\rm ind}_{\zkz}( D_{(E_1\oplus E_2,F_1\oplus F_1)})
\end{equation*}
where the first index is taken on the $\zkz$-manifold $(Q\dot{\cup}Q,P\dot{\cup}P)$ while the latter is over $(Q,P)$.   \par
Next, consider the bordism relation.  Using Theorem \ref{borInv} (i.e., the $\zkz$-index is a bordism invariant of $\zkz$-manifolds), we conclude that the map passes to bordism equivalence classes.  \par
Finally, we prove that the $\zkz$-index is invariant under $\zkz$-vector bundle modification.  To fix notation, let $((Q,P),(E,F),f)$ be a $\zkz$-cycle and $(W,V)$ a ${\rm spin^c}$ $\zkz$-vector bundle with even dimensional fibers.  Moreover, we note that the $\zkz$-vector bundle modification of $((Q,P),(E,F),f)$ by $(W,V)$ will be denoted by $((Q^W,P^V),(E^W,F^V),f\circ \pi)$ and $\pi:(Q^W,P^V) \rightarrow (Q,P)$ is a $\zkz$-fiber bundle and that the fibers of $\pi$ are even dimensional spheres.  Using the disjoint union operation, we need only consider the case when $Q$ is connected and hence that each fiber of $\pi$ is $S^{2n}$ for some fixed natural number $n$. Finally, we denote by $\pi^{S^{2n}}$ the direct image map from $K^0(\tilde{Q}^W)$ to $K^0(\tilde{Q})$ (see Equation \ref{multZkzProp} on p.g. 8). \par
Two facts about the direct image map are required:
\begin{enumerate}
\item $\pi^{\tilde{Q}^W}_{!}= \pi^{\tilde{Q}}_{!} \circ \pi^{S^{2n}}_{!}$
\item $\pi^{S^{2n}}_{!}([(E^W,F^V)]) = [(E,F)]$
\end{enumerate}
The first of these facts is a special case of Equation \ref{multZkzProp}, while the second follows from the relationship between the direct image map and the Thom isomorphism in K-theory.  In particular, $[(E^W,F^V)]$ is the image of $(E,F)$ under the Thom isomorphism. \par
Using these facts, we have that 
\begin{eqnarray*}
\Phi(((Q^W,P^V),(E^W,F^V),f\circ \pi)) & = & \pi^{\tilde{Q}^W}_{!}( [(E^W, F^V)] ) \\
& = & \pi^{\tilde{Q}}_{!}( \pi^{S^{2n}}_{!}( [(E^W,F^V)]) ) \\
& = & \pi^{\tilde{Q}}_{!}( [(E,F)]) \\
& = & \Phi(((Q,P),(E,F),f))
\end{eqnarray*}
This proves the invariance of the $\zkz$-topological index under $\zkz$-vector bundle modification.  Thus, $\Phi$ defines a map between $K_0(pt;\zkz)$ and $\zkz$. \par
That it is a group homomorphism follows from the fact that it respects disjoint union (i.e., the group operation) and also respects the operation of taking the opposite ${\rm spin^c}$ structure (i.e., the inverse operation in the group). \par
We now show, using the Bockstein exact sequence, that $\Phi$ is a group isomorphism.  The Bockstein sequence, in the special case of $X=pt$, has the form 
\begin{center}
$\begin{CD}
K_1(pt;\zkz) @>>> K_0(pt) @>>> K_0(pt) @>>> K_0(pt;\zkz) @>>> K_1(pt) \\
@VVV  @VVV @VVV @VVV @VVV \\
0 @>>>  \field{Z} @>>> \field{Z} @>>> \zkz @>>> 0  
\end{CD}$
\end{center}
where we have used the following three facts: 
\begin{enumerate}
\item $K_0(pt) \cong \field{Z}$ via the map which takes an even Baum-Douglas cycle to its index; 
\item $K_1(pt) \cong 0$
\item Multiplication by $k$ is injective in the group $K_0(pt)\cong\field{Z}$. 
\end{enumerate}
Moreover, the commutativity of the connecting maps between these exact sequences follows from the following facts:
\begin{enumerate}
\item The topological index respects the group operation, and hence respects multiplication by $k$ in $K_0(pt)$.
\item The topological Freed-Melrose index of a $\zkz$-cycle which is in the image of $r$ (i.e., is of the form $(M,\emptyset),(E,\emptyset),f)$) is the topological index of the Baum-Douglas cycle $(M,E,f)$ reduced ${\rm mod}\;k$.  
\end{enumerate}
The Five Lemma applied to these exact sequences leads to the fact that $\Phi$ is an isomorphism.
\end{proof}      
\end{ex}
\section{Direct Limits and $K_*(X;G)$}
We now use the geometric models for $K_*(X;\field{Z})$ and $K_*(X;\zkz)$ and inductive limits to construct geometric models for $K$-homology with coefficients in any abelian group.  The general idea is similar to that of Chapter 13 of \cite{Ror};  we leave the details to the interested reader. 
\par
To aid such a reader, we note that for $K(X;\field{Z}^k)$, we use $k$-tuples of Baum-Douglas cycles and for each $\field{Z}/n_s\field{Z}$ ($s=1,\ldots, r$), we use $\field{Z}/n_s\field{Z}$-cycles.  For inductive limits of groups of the form required, we need to consider group homomorphisms between $\field{Z}/n\field{Z}$ and $\zkz$.  Therefore, we need to construct maps which take $\field{Z}/n\field{Z}$-manifolds to $\zkz$-manifolds.
\par
We follow the construction in \cite{MS} to define the required maps.  For any $\zkz$-cycle, $((Q,P),(E,F),f)$, and $l\in \field{N}$, we define $r_l((Q,P),(E,F),f)$ to be the $\field{Z}/(l\cdot k)\field{Z}$-cycle $((l\cdot Q,P),(E,F),f)$.  Now for any $\field{Z}/(l\cdot k)\field{Z}$-cycle, $((Q,P),(E,F),f)$, we define $R_l((Q,P),(E,F),f)$ to be the $\zkz$-cycle defined by $((Q,l\cdot P),(E,F),f)$.  (The observant reader will note that the second of these maps is, in fact, only defined up to $\zkz$-bordism.)  Thus, if $G$ is an abelian group which is an inductive limit of copies of $\zkz$ and $\field{Z}$, then cycles for $K_*(X;G)$ can be constructed using $\zkz$-cycles and $\field{Z}$-cycles with their standard relations and an additional relation coming from the maps in the inductive limit.     
\par
The process of constructing geometric cycles using inductive limits may seem a priori artificial.  However, for specific coefficients group, we can reformulate such models at the level of cycles.  The case of $\field{Q}/\field{Z}$ is prototypical.  In this case, the cycles we have constructed for $K_*(X;\field{Q}/\field{Z})$ can be reformulated as follows.  
\begin{define} \label{qzCycle}
Let $X$ be a finite CW-complex.  Then a $\field{Q}/\field{Z}$-cycle over $X$ is a triple, $((Q,P,\pi),(E,F,\theta),(f_Q,f_P))$, where 
\begin{enumerate}
\item $Q$ and $P$ are compact ${\rm spin^c}$-manifolds (the former with boundary) and $\pi$ is trivial covering map $\partial Q \rightarrow P$.
\item $E$ and $F$ are vector bundles over $Q$ and $P$ respectively and $\theta$ is a lift of $\pi$, which is an isomorphism between $E|_{\partial Q}$ and $\pi^*(F)$.
\item $f_Q: Q \rightarrow X$ and $f_P:P \rightarrow X$ are continuous and $f_P \circ \pi = f_Q|_{\partial Q}$.
\end{enumerate}
\end{define}
Objects of this form (i.e., manifolds whose boundary is a finite trivial cover) have appeared in index theory before.  Most notably, a $\field{Q}/\field{Z}$-valued index is constructed for such objects in \cite{APS2}.  More generally, an $\field{R}/\field{Z}$-invariant (the $\rho$-invariant) is constructed and an index theorem for flat vector bundles is proved in \cite{APS3}.  This construction is done in K-theory with coefficients in $\field{R}/\field{Z}$.  As such, we can view Definition \ref{qzCycle} as a starting point for the construction of a model for K-homology with coefficients in $\field{R}/\field{Z}$.  Such a model should have applications to the both the $\eta$ and $\rho$-invariant.  The desire for such a construction in geometric K-homology is stated in Remark 6.12 of \cite{HReta}.

\section{Outlook}
The reader who is is familiar with K-homology will know that there is also an analytic model using Fredholm modules as cycles (see Chapter 8 \cite{HR}).  Moreover, there is natural map (constructed by Baum and Douglas) from geometric K-homology to analytic K-homology (see \cite{BHS}).  For a finite CW-complex, this map was shown to be an isomorphism by Baum and Douglas (also see \cite{BHS}).  Thus, the reader may ask if there is a similar construction from our geometric model for $K_*(X;\zkz)$ to an analytic realization of K-homology with coefficients in $\zkz$.  The answer is affirmative.  The second paper in this sequence \cite{Dee2} will deal with the construction of a suitable analytic realization of K-homology with coefficients and the construction of a map analogous to the one constructed by Baum and Douglas.  This paper uses results in \cite{Ros1} and \cite{Sch} in a fundamental way. 
\vspace{0.5cm} \\ 
{\bf Acknowledgments} \\
I would like to thank my PhD supervisor, Heath Emerson, for useful discussions on the content and style of this document.  In addition, I thank Nigel Higson, Jerry Kaminker, John Phillips, Ian Putnam, and Thomas Schick for discussions.  I would also like to thank the reviewer for a number of useful suggestions.  This work was supported by NSERC through a PGS-Doctoral award.

\vspace{0.25cm}
Email address: rjdeeley@uni-math.gwdg.de \vspace{0.25cm} \\
{ \footnotesize MATHEMATISCHES INSTITUT, GEORG-AUGUST UNIVERSIT${\rm \ddot{A}}$T, BUNSENSTRASSE 3-5, 37073 G${\rm \ddot{O}}$TTINGEN, GERMANY}

\end{document}